\documentclass{amsart}
\usepackage{amsfonts}
\usepackage{amsmath,amssymb}
\usepackage{amsthm}
\usepackage{amscd}
\usepackage{graphics}
\usepackage{graphicx}
\theoremstyle{definition}
{
\newtheorem{Def}{{\rm Definition}}
\newtheorem{Ex}{{\rm Example}}
\newtheorem{Rem}{{\rm Remark}}

}
\newtheorem{Cor}{Corollary}
\newtheorem{Prop}{Proposition}
\newtheorem{Thm}{Theorem}
\newtheorem{Lem}{Lemma}

\begin{document}
\title[Explicit smooth maps on 7-dimensional manifolds of codimension -3]{A note on explicit smooth maps on 7-dimensional manifolds into the 4-dimensional Euclidean space}
\author{Naoki Kitazawa}
\address{Institute of Mathematics for Industry, Kyushu University, 744 Motooka, Nishi-ku Fukuoka 819-0395, Japan \\
TEL (Office): +81-92-802-4402 \\
FAX (Office): +81-92-802-4405}
\email{n-kitazawa@imi.kyushu-u.ac.jp}
\maketitle
\begin{abstract}

A {\it fold} map is a smooth map at each singular point of which it is represented as the product map of a Morse function and the identity map on an open ball. A {\it special generic map} is a fold map such that the Morse function can be taken as a natural height function on an unit disk. The class of special generic maps includes a Morse function with exactly two singular points on a closed manifold, characterizing a sphere topologically (except $4$-dimensional cases) as the Reeb's theorem shows, and canonical projections of unit spheres.

It has been known that so-called {\it exotic} spheres do not admit special generic maps into Euclidean spaces whose dimensions are sufficiently high and smaller than the dimensions of the spheres. Exotic $7$-dimensional homotopy spheres do not admit special generic maps into the $4$-dimensional Euclidean space for example. We can easily obtain special generic maps on fundamental manifolds such as ones represented as connected sums of products of two standard spheres and in considerable cases, smooth manifolds resembling topologically them and different from them do not admit special generic maps. These interesting results are due to studies of Saeki, Sakuma and Wrazidlo since the 1990s.

In the present paper, we present new results on explicit smooth maps including fold maps on $7$-dimensional manifolds into the $4$-dimensional Euclidean space and meanings in algebraic topology and differential topology of manifolds. Moreover, the author obtained related results before motivated by the studies before and they are reviewed in the presentation of the new results. We also present new discussions and results related to the results for $7$-dimensional manifolds and maps on them for fold maps between manifolds of general dimensions.
\end{abstract}


\maketitle
\section{Introduction and fold maps.}
\label{sec:1}
The present paper is on explicit smooth maps including {\it fold} maps on $7$-dimensional closed manifolds into the $4$-dimensional Euclidean space and general discussions on fold maps between manifolds of general dimensions. {\it Fold} maps are, in short, higher dimensional versions of Morse functions and important tools in studying geometric properties of manifolds in the branch of the singularity theory of differentiable maps and applications to geometry of manifolds as {\it Morse} functions are in so-called Morse theory: Morse theory is in a sense regarded as specific classical theory of the branch. $7$-dimensional manifolds form an interesting class of manifolds. Especially, Milnor's discovery of $7$-dimensional {\it exotic} spheres, which are homotopy spheres having differentiable structures different from those of unit spheres \cite{milnor}, has made this class more attractive and this class is still attractive. Recent attractive studies on these manifolds are \cite{crowleyescher}, \cite{crowleynordstrom} and \cite{kreck} for example: note also that these manifolds are closed and simply-connected.

Let us introduce terminologies and notation on topological and geometric objects in the present paper. ${\mathbb{R}}^k$ denotes the $k$-dimensional Euclidean space, endowed with the standard Euclidean metric. $||x|| \geq 0$ denotes the distance between $x$ and the origin $0$ in the Euclidean space.
A {\it homotopy} sphere means a manifold which is homeomorphic to a sphere. $S^k:=\{x \in {\mathbb{R}}^{k+1} \mid ||x||=1.\}$ denotes the $k$-dimensional unit sphere and we call a (smooth) homotopy sphere which is a copy, or equivalently, diffeomorphic to this a ({\it $k$-dimensional}) {\it standard sphere}.
An {\it exotic} sphere is a smooth homotopy sphere which is not diffeomorphic to any standard sphere. $D^k:=\{x \in {\mathbb{R}}^{k} \mid ||x||=1.\}$ denotes the $k$-dimensional unit disk and we call a smooth manifold which is a copy, or equivalently, diffeomorphic to this a ({\it k-dimensional}) {\it standard disk}.
Throughout the present paper, manifolds and maps between them are smooth and of class $C^{\infty}$ unless otherwise stated. Diffeomorphisms on manifolds are always smooth and the {\it diffeomorphism group} of a manifold is the group of all diffeomorphism on it. For bundles whose fibers are manifolds, the structure groups are subgroups of the diffeomorphism groups unless otherwise stated. They are {\it smooth} bundles.
An {\it $X$-bundle} means a (smooth) bundle whose fiber is (diffeomorphic) to a topological space (resp. smooth manifold) $X$.

A {\it linear} bundle is a bundle whose fiber is a unit sphere or a unit disk and whose structure group acts linearly on its fiber in a canonical way.
A {\it singular} point of a differentiable map $c$ is a point at which the rank of the differential of the map is smaller than both the dimensions of the manifolds of the domain and the target. The set $S(c)$ of all singular points is the {\it singular set} of the map. The {\it singular value set} of the map is the image of the singular set. The {\it regular value set} of the map is the complementary set of the singular value set in the manifold of the target. A {\it singular {\rm (}regular{\rm )} value} is a point in the singular (resp. regular) value set.

\subsection{Fold maps.}
A smooth map between an $m$-dimensional smooth manifold no boundary into an $n$-dimensional smooth manifold with no boundary is said to be a {\it fold} map if at each singular point $p$, the map is represented as
$$(x_1, \cdots, x_m) \mapsto (x_1,\cdots,x_{n-1},\sum_{k=n}^{m-i}{x_k}^2-\sum_{k=m-i+1}^{m}{x_k}^2)$$
for some suitable coordinates and an integer $0 \leq i(p) \leq \frac{m-n+1}{2}$.
For a fold map, the following three are fundamental properties.
\begin{itemize}
\item For any singular point $p$, $i(p)$ is unique.
\item The set consisting of all singular points of a fixed index of the map is a smooth closed submanifold of the $m$-dimensional manifold of the domain whose dimension is $n-1$ and which has no boundary.
\item The restriction map to the singular set is a smooth immersion.
\end{itemize}
{$i(p)$ is said to be the {\it index} of $p${\rm ). A Morse function on a manifold with no boundary is a fold map.
For fundamental study of Morse functions and fold maps, we have various books and articles.
For the theory of Morse functions or Morse theory, \cite{milnor2} is well-known and \cite{milnor3} emphasizes algebraic topological and differential topological viewpoints. \cite{golubitskyguillemin} is on the singularity theory of differentiable maps and singularities appearing in Morse functions and fold maps are explained (where fold maps are not explicitly defined). \cite{thom} and \cite{whitney} are pioneering papers on so-called {\it generic} smooth maps from manifolds whose dimensions are greater than or equal to $2$ into the plane and fold maps are regarded as such maps without so called {\it cusp} points, which form finite subsets of the manifolds.
\subsection{Special generic maps}
A {\it special generic} map is a fold map such that the index of each singular point is $0$. A Morse function on a closed manifold with just two singular points, characterizing a sphere topologically (except $4$-dimensional cases) as the Reeb's theorem \cite{reeb} states, and the canonical projection of a unit sphere are simplest special generic maps. It is an elementary exercise on smooth manifolds, smooth functions and maps and Morse theory to check that the canonical projection is special generic.
 It is an interesting fact that special generic maps restrict the topologies and the differentiable structures of the manifolds admitting them strongly in considerable cases. For example, exotic spheres of dimension $m>3$ do not admit special generic maps into
${\mathbb{R}}^{m-3}$, ${\mathbb{R}}^{m-2}$ and ${\mathbb{R}}^{m-1}$. For integers $m>n \geq 1$, on an $m$-dimensional manifold represented as a connected sum of manifolds represented as products of two standard spheres at least one of whose dimension is smaller than $n$ and positive, we can easily construct a special generic map into ${\mathbb{R}}^n$. However, for example, it is known that $4$-dimensional manifolds homeomorphic to these manifolds and not diffeomorphic to them exist and that they do not admit special generic maps into ${\mathbb{R}}^3$. For these studies, see \cite{nishioka}, \cite{saeki}, \cite{saeki2}, \cite{saekisakuma}, \cite{saekisakuma2} and \cite{wrazidlo} for example.

\subsection{Explicit fold maps on $7$-dimensional manifolds and application to algebraic and differential topology of manifolds.}
We can see that every $m$-dimensional homotopy sphere admits a fold map into ${\mathbb{R}}^n$ satisfying $m \geq n \geq 1$ by the theory on existence of fold maps via homotopy principle by \cite{eliashberg} and \cite{eliashberg2}. The author previously constructed explicit fold maps into ${\mathbb{R}}^4$ on every $7$-dimensional homotopy sphere. The author has also discovered that the topologies of the singular value sets and the differentiable structures are closely related (Corollary \ref{cor:1}).

Motivated by these studies, in the present paper, we study explicit smooth maps, especially, fold maps into ${\mathbb{R}}^4$ on $7$-dimensional manifolds, which produce interesting problems on algebraic topology and differential topology of manifolds. For example, we will see new explicit fold maps into ${\mathbb{R}}^4$ on $7$-dimensional closed and simply-connected manifolds of a class about which comprehensive algebraic topological studies have been demonstrated by Kreck (\cite{kreck}) for example, and explore meanings in algebraic topological or differential topological theory of manifolds.

\subsection{The content of the paper.}
The organization of the paper is as the following.
In the next section, we review fundamental properties and known results on special generic maps and some other classes of fold maps and manifolds including $7$-dimensional ones admitting them: Corollary \ref{cor:1} is one of most meaningful results with respect to the main theme of the present paper implying that types of fold maps into ${\mathbb{R}}^4$ restrict the differentiable structures of the $7$-dimensional homotopy spheres. The author has obtained this before and presented in \cite{kitazawa2} for example. The last section is devoted to new algebraic topological and differential topological results on fold maps on $7$-dimensional closed manifolds and the presentation of new explicit fold maps on these manifolds.

Theorems \ref{thm:5} (\ref{thm:4})--\ref{thm:8} are main theorems of the present paper.
Theorem \ref{thm:4} is on characteristic classes ({\it 1st Pontryagin classes}) and Theorems \ref{thm:5} and \ref{thm:6} present a family of new explicit fold maps with information of $7$-dimensional closed and simply-connected manifolds admitting these maps. We also show Theorem \ref{thm:7}.

We add comments on Theorems \ref{thm:5} and \ref{thm:6} reviewing \cite{kitazawa6}. Theorems \ref{thm:5} and \ref{thm:6} are regarded as results obtained by explicitizing abstract results of \cite{kitazawa6} to see $7$-dimensional manifolds via fold maps into ${\mathbb{R}}^4$. \cite{kitazawa6} is on homology groups, cohomology rings and other invariants for {\it Reeb spaces} of fold maps, which are defined as the spaces of all connected components of preimages of maps and defined as the quotient spaces of the manifolds.

Reeb spaces are in suitable cases inherit invariants for the manifolds such as homology groups and cohomology rings considerably.

In the paper, essentially, the Reeb spaces of fold maps such that preimages of regular values are disjoint unions of standard spheres and that the restrictions to the singular sets are embeddings are systematically studied. The fact that for sufficiently high dimensional manifolds of the domains, Reeb spaces inherit invariants before well is presented. Some new facts on sufficiently high dimensional manifolds admitting these maps are also explained.
However, most of the content is on the Reeb spaces. We do not need to study the paper precisely to study the present paper.

On the other hand, the present paper is devoted mainly to $7$-dimensional manifolds admitting maps satisfying the properties explained here and the spaces of the target spaces are $4$-dimensional: the dimensions of the manifolds of the domains are not sufficiently large in most cases. In proving the theorems, we prove key propositions such as Proposition \ref{prop:6}. Discussions and propositions including this proposition are new: we study not only homology groups and cohomology rings of Reeb spaces of maps but also those of the manifolds admitting these maps precisely.

We close the present paper by presenting Theorem \ref{thm:8} and Remark \ref{rem:2}. Theorem \ref{thm:8} generalizes manifolds, submanifolds and homology and cohomology groups in Theorem \ref{thm:5}.


\section{Special generic maps and round fold maps.}
\label{sec:2}
Throughout this paper, $M$ is an $m$-dimensional closed and connected manifold, $n<m$ is a positive integer and $f:M \rightarrow {\mathbb{R}}^n$ is a smooth map unless otherwise stated.
\subsection{Fundamental differential topological properties of special generic maps and manifolds admitting them.}
We introduce fundamental properties and propositions on special generic maps. See articles referred in section \ref{sec:1} and as related other articles, see also \cite{kitazawa}, \cite{kitazawa2}, \cite{kitazawa3}, \cite{kitazawa4} ,\cite{kitazawa5}, \cite{kitazawa6} and \cite{saekisuzuoka} for example.
\begin{Prop}
\label{prop:1}
Let $m>n \geq 1$ be integers.
\begin{enumerate}
\item For an $m$-dimensional closed and connected manifold $M$, if it admits a special generic map into ${\mathbb{R}}^n$, then it is regarded as the composition of a suitable smooth map onto an $n$-dimensional compact and connected manifold satisfying the following properties with a smooth immersion into ${\mathbb{R}}^n$.
\begin{enumerate}
\item The boundary of the compact manifold coincides with the singular value set.
\item On the preimage of the interior of this map, the map onto this interior gives a smooth $S^{m-n}$-bundle.
\item On the preimage of a small collar neighborhood of the manifold of the target, the composition of the map onto this neighborhood with the canonical projection onto the boundary gives a linear $D^{m-n+1}$-bundle.
\end{enumerate}
\item For an $n$-dimensional compact and connected manifold we can smoothly immerse into ${\mathbb{R}}^n$, there exists an $m$-dimensional closed and connected manifold admitting a special generic map into ${\mathbb{R}}^n$ such that the $n$-dimensional manifold explained in the previous statement is diffeomorphic to the given $n$-dimensional manifold and that the special generic map is represented as the composition of the smooth surjection as in the previous statement with a smooth immersion of the $n$-dimensional manifold into ${\mathbb{R}}^n$.

\end{enumerate}
\end{Prop}
In this proposition, for a special generic map $f:M \rightarrow {\mathbb{R}}^n$ on an $m$-dimensional closed and connected manifold, the $n$-dimensional compact manifold we can immerse into ${\mathbb{R}}^n$ is denoted by $W_f$ and the smooth surjection by $q_f:M \rightarrow W_f$.
\begin{Ex}
\label{ex:1}
Let $m>n \geq 2$ and $l>0$ be integers. An $m$-dimensional closed and connected manifold $M$ represented as a connected sum of manifolds of a family $\{S^{l_j} \times S^{m-l_j}\}_{j=1}^{l}$ satisfying $1 \leq l_j \leq n-1$ admits a special generic map
$f:M \rightarrow {\mathbb{R}}^n$ with the following properties where the connected sum is considered in the smooth category.
\begin{enumerate}
\item $W_f$ is represented as a boundary connected sum of manifolds of a family $\{S^{l_j} \times D^{n-l_j}\}$ where the boundary connected sum is considered in the smooth category.
\item $f {\mid}_{S(f)}$ is an embedding.
\item The $S^{m-n}$-bundle and the $D^{m-n+1}$-bundle in Proposition \ref{prop:1} are trivial bundles.
\end{enumerate}
More precisely, in the case $n=2$, a closed manifold is represented as a connected sum of total spaces of smooth $S^{m-1}$-bundles over $S^1$ if and only if it admits a special generic map into ${\mathbb{R}}^2$.
This holds in the case $n=3$ under the constraint that the fundamental group of the manifold is free. The connected sums are considered in the smooth category as before. See \cite{saeki}, \cite{saeki2} and \cite{saekisakuma} for example.
\end{Ex}
\begin{Prop}
\label{prop:2}
For a special generic map $f:M \rightarrow {\mathbb{R}}^n$ on an $m$-dimensional closed and connected manifold $M$ satisfying $m>n \geq 1$, $M$ is bounded by a compact and connected PL manifold $W$ which collapses to $W_f$.
Especially, if $m-n=1,2,3$, then we can take $W$ as a smooth manifold.
\end{Prop}
\begin{Prop}
\label{prop:3}
For a special generic map $f:M \rightarrow {\mathbb{R}}^n$ on an $m$-dimensional closed and connected manifold $M$ satisfying $m>n \geq 1$, the surjection $q_f$ induces isomorphisms between the homology groups, cohomology rings and homotopy groups of degree $j \leq m-n$. Moreover, in the situation of Proposition \ref{prop:2}, $q_f$ is represented as the composition of the inclusion $i:M \rightarrow W$ with a natural map to $W_f$ giving a collapsing.
\end{Prop}

\begin{Prop}
\label{prop:4}
The following two are equivalent for a special generic map $f:M \rightarrow {\mathbb{R}}^n$ on an $m$-dimensional closed and connected manifold satisfying $m>n \geq 1$.
\begin{enumerate}
\item $M$ is a homotopy sphere.
\item The manifold $W_f$ is contractible.
\end{enumerate}
\end{Prop}

\begin{Thm}[\cite{saeki}, \cite{saeki2} and \cite{wrazidlo}]
\label{thm:1}
Let $a>b \ge1 $ be integers satisfying $a>3$ and $b=a-3,a-2,a-1$.
Every $a$-dimensional exotic sphere does not admit a special generic map into ${\mathbb{R}}^b$.
Furthermore, $7$-dimensional oriented smooth homotopy spheres of $14$ types of all the $28$ types do not admit special generic maps into ${\mathbb{R}}^3$.
\end{Thm}
\begin{Prop}[\cite{nishioka}]
\label{prop:5}
For a special generic map $f:M \rightarrow {\mathbb{R}}^n$ on an $m$-dimensional closed and connected manifold $M$ such that $H_1(M;\mathbb{Z})$ is zero satisfying $m>n \geq 3$, $H_1(W_f;\mathbb{Z})$ is zero and $H_{n-2}(W_f;\mathbb{Z})$ and $H_{n-1}(W_f;\mathbb{Z})$ are free.
\end{Prop}

\subsection{Round fold maps.}
We review {\it round} fold maps, introduced by the author first in \cite{kitazawa}, \cite{kitazawa2} and \cite{kitazawa3}. $\mathbb{N}$ denotes the set of all positive integers.
\begin{Def}
\label{def:1}
For a fold map $f:M \rightarrow {\mathbb{R}}^n$ satisfying $m \geq n \geq 2$, $f$ is said to be a {\it round} fold map if the restriction to the singular set is an embedding and the image is concentric spheres: more precisely, for
a suitable diffeomorphism $\phi$ and an integer $l>0$ $(\phi \circ f)(S(f))=\{x \in {\mathbb{R}}^n \mid 1 \leq ||x|| \in \mathbb{N} \leq l\}$ holds.
\end{Def}
\begin{Ex}
\label{ex:2}
The canonical projection of the unit sphere of dimension $m \geq n \geq 2$ into ${\mathbb{R}}^n$ and some special generic maps on homotopy spheres are round fold maps.
\end{Ex}
\begin{Def}
\label{def:2}
For each connected component $C$ of the singular value set of a round fold map $f:M \rightarrow {\mathbb{R}}^n$, there exists a small closed tubular neighborhood and if we consider the restriction of $f$ to the preimage of the closed tubular neighborhood and composition of this with a canonical projection to $C$, then we have a smooth bundle over $C$. If this bundle is trivial for each $C$, then $f$ is said to be {\it componentwisely trivial}.
\end{Def}
Maps in Example \ref{ex:2} are componentwisely trivial.
\begin{Thm}[\cite{kitazawa}, \cite{kitazawa2} and \cite{kitazawa4}]
\label{thm:2}
Let $m>n \geq 2$ be integers. We can construct a componentwisely trivial round fold map on a manifold represented as a connected sum of total spaces of $S^{m-n}$-bundles over $S^n$ into ${\mathbb{R}}^n$ satisfying the following properties where the connected sum is taken in the smooth category.
\begin{enumerate}
\item Preimages of regular values are disjoint unions of standard spheres.
\item In ${\mathbb{R}}^n$, the number of connected components of a preimage increases as we go into a point in the just one connected component diffeomorphic to an open ball of the regular value set starting from the unbounded connected component of the regular value set. More rigorously, in Definition \ref{def:1}, for each point in $p$ satisfying $1<||p|| \neq \mathbb{N}<l+1$, the number of connected components of each point $q \in {\mathbb{R}}^n$ satisfying $||q||=||p||-1$ is greater than that of $p$ by $1$.
\end{enumerate}
\end{Thm}
\begin{Rem}
We can define the notion of a {\it round} fold map for $n=1$ and we have a theorem similar to Theorem \ref{thm:2}. We do not need in the present study. Consult \cite{kitazawa4} for example.
\end{Rem}
Thanks to the theorem with \cite{eellskuiper} and \cite{milnor} for example, we obtained this before.
\begin{Cor}[\cite{kitazawa2} and so on.]
\label{cor:1}
Every $7$-dimensional homotopy sphere admits a round fold map as in Theorem \ref{thm:2} into ${\mathbb{R}}^4$. Moreover, the number of connected components of the singular value restrict the differentiable structure of the oriented homotopy sphere as follows {\rm (}there are $28$ types of oriented smooth homotopy spheres as explained in Theorem \ref{thm:1}{\rm )}.
\begin{enumerate}
\item The number is $1$ if and only if the homotopy sphere is a standard sphere.
\item The number is $2$ if and only if the homotopy sphere is one of 16 types of the 28 types {\rm (}oriented smooth homotopy spheres of these $16$ types are represented as total spaces of linear $S^3$-bundles over $S^4$ and a standard sphere is a smooth homotopy sphere of one of these 16 types{\rm )}.
\item Every $7$-dimensional oriented smooth homotopy sphere admits a round fold map into ${\mathbb{R}}^4$ as in Theorem \ref{thm:2} such that the number of connected components of the singular value set is $3$ {\rm (}oriented homotopy spheres of the $28-16=12$ types of the 28 types are not represented as total spaces of linear $S^3$-bundles over $S^4$ but represented as connected sums of these total spaces{\rm )}.
\end{enumerate}
\end{Cor}
\begin{figure}
\includegraphics[width=50mm]{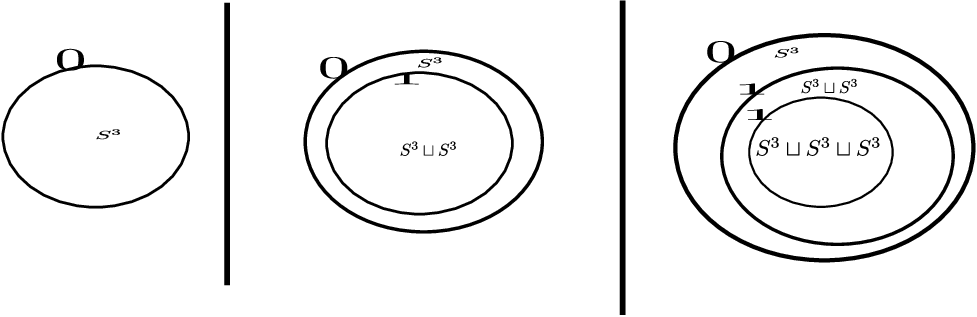}
\caption{The images of round fold maps in Corollary \ref{cor:1} : the figure in the left represents the first case, the figure in the center represents the second case and the figure in the right represents the third case.}
\label{fig:1}
\end{figure}
FIGURE \ref{fig:1} represents the images of round fold maps into ${\mathbb{R}}^4$. Circles represent singular value sets, which are disjoint unions of $3$ -dimensional standard spheres. $S^3$, $S^3 \sqcup S^3$ and $S^3 \sqcup S^3 \sqcup S^3$ represent preimages of regular values in the connected components of regular value sets they are included. $0$ and $1$ are for indices of singular points (values).
\section{New results on topological properties of explicit fold maps on $7$-dimensional manifolds and new examples of fold maps on $7$-dimensional manifolds.}
\subsection{On characteristic classes of linear bundles and manifolds.}
The following is well-known. See \cite{milnor} and see also \cite{crowleyescher} for example. Note that smooth $S^3$-bundles are linear by the theory \cite{hatcher}. Note also that we consider {\it oriented} linear bundles in the following: structure groups are rotation groups and fibers are oriented. See \cite{milnorstasheff} and \cite{steenrod} for general theory of vector bundles, (oriented) linear bundles and bundles whose fibers are more general spaces.
For {\it Euler classes} of oriented linear bundles, {\it {\rm (}$k$-th{\rm )} Steifel-Whitney classes} of linear bundles and {\it {\rm (}$k$-th{\rm )} Pontryagin classes} of such bundles are defined as cohomology classes of base spaces of the bundles and see \cite{milnorstasheff}. For notation on algebraic systems, $\mathbb{Z}$ denotes the ring of integers for example.
\begin{Thm}
\label{thm:3}
\begin{enumerate}
\item Linear $S^3$-bundles over $S^4$, which are linear and obtained by gluing two copies of trivial linear $S^3$-bundles over $D^4$ by a bundle isomorphism between the linear bundles defined on the boundaries, are classified by ${\pi}_{3}(SO(4))$ (if they are oriented), isomorphic to $\mathbb{Z} \oplus \mathbb{Z}$. More precisely, these oriented linear bundles are completely classified by the Euler classes and the cohomology classes obtained by certain procedures so that the 1st Pontryagin classes are twice these classes.
\item Fix a generator of $H^4(S^4;\mathbb{Z})$.
The total space $M_a$ of an oriented linear $S^3$-bundle over $S^4$ such that the Euler class is $a$ times the generator is $2$-connected and $H_3(M_{a};\mathbb{Z})$ is isomorphic to $\mathbb{Z}/|a|\mathbb{Z}$.
\item Fix a generator $\nu \in H^4(S^4;\mathbb{Z})$. Total spaces of oriented linear $S^3$-bundles over $S^4$ such that the Euler classes are $0$ are (co)homologically regarded as $S^4 \times S^3$. More precisely, between the total spaces of these two bundles, there exists an isomorphism of the integral cohomology rings.
These linear bundles are completely classified by the 1st Pontryagin classes and the classes are represented by $4k\nu \in H^4(S^4;\mathbb{Z})$ for an integer $k$: we can realize such a bundle for each integer $k$ and a trivial linear bundle is for the case $k=0$.
Moreover, the 1st Pontryagin classes of the total spaces {\rm (}, which are canonically oriented,{\rm )} are obtained as the pull-backs via the canonical projections.
\end{enumerate}
\end{Thm}

\begin{Lem}
\label{lem:1}
For a special generic map $f:M \rightarrow {\mathbb{R}}^n$ with $m-n=1,2,3$ and $n=1,2,3,4$, the 1st Pontryagin class of $M$ vanishes.
\end{Lem}
\begin{proof}
By virtue of Proposition \ref{prop:2}, $M$ is bounded by a smooth, compact and orientable manifold $W$ simple homotopy equivalent to a compact $n$-dimensional manifold. The 1st Pontryagin class of $W$ is zero since the class is of degree $4$ and $W$ has the homotopy type of an ($n-1$)-dimensional polyhedron ($n-1 \leq 3$). The restriction of the class to $M$ completes the proof.
\end{proof}
Hereafter, for a $k$-dimensional oriented closed manifold, the fundamental class is the homology class of degree $k$ compatible with the orientation. If the coefficient ring is isomorphic to $\mathbb{Z}/2\mathbb{Z}$ and of order $2$, then this class does not depend on the orientation.
A homology class of a compact manifold is {\it represented} by a closed submanifold with no boundary if we can orient the submanifold and the class and the value of the homomorphism induced by the inclusion at the fundamental class agree. In the present study, we discuss these notions only in the smooth category and we can also discuss these notions in the PL and topological categories for example.

We have the following result.
\begin{Thm}
\label{thm:4}
For a smooth map $f:M \rightarrow {\mathbb{R}}^4$ on an $m$-dimensional closed, connected and oriented manifold $M$ into ${\mathbb{R}}^4$ satisfying $m \geq 7$, let $F$ be a connected component of the preimage of a regular value diffeomorphic to
$S^{m-n}$ and the homology class $[F]$ represented by $F$ satisfies the following properties.
\begin{enumerate}
\item {\rm (}On the order{\rm )} The order of the homology class $[F]$ where the coefficient ring is $\mathbb{Z}$ is infinite.
\item {\rm (}On the divisibility{\rm )} For any integer $k>1$ and any homology class $[F^{\prime \prime}]$, the class $[F]$ is not represented as $k[F^{\prime \prime}]$.
\end{enumerate}
In this situation, for the Poincar\'e dual ${\rm PD}([F])$ to $[F]$ in $H^4(M;\mathbb{Z})$, and any element represented as $4l{\rm PD}([F])$ for an integer $l$, there exist an $m$-dimensional closed, connected and oriented manifold $M^{\prime}$ and a smooth map $f^{\prime}:M^{\prime} \rightarrow {\mathbb{R}}^4$ satisfying the following properties.
\begin{enumerate}
\item There exists an isomorphism $i_{M,M^{\prime}}$ between the cohomology rings of $H^{\ast}(M;\mathbb{Z})$ and $H^{\ast}(M^{\prime};\mathbb{Z})$ such that the first Pontryagin class of $M^{\prime}$ is the value of the isomorphism at the sum of the first Pontryagin class of $M$ and $4l{\rm PD}([F])$.
\item For $f^{\prime}$, there exists a connected component $F^{\prime}$ of the preimage of a regular value diffeomorphic to $S^{m-n}$.
\item The restriction of $f$ to a small closed tubular neighborhood of the connected component $F$ of the preimage gives a trivial linear $F$-bundle over a $4$-dimensional standard disk $D$ smoothly embedded in ${\mathbb{R}}^4$ of the target. The restriction of $f^{\prime}$ to a small closed tubular neighborhood of the connected component $F^{\prime}$ of the preimage diffeomorphic to $F$ before gives a trivial linear $F^{\prime}$-bundle over the $4$-dimensional standard disk $D$. If we consider the complements and the restrictions, then for the restrictions $f_{F}$ and ${f^{\prime}}_{F^{\prime}}$ to the closed tubular neighborhoods, there exists a diffeomorphism $\Phi$ satisfying ${f^{\prime}}_{F^{\prime}} \circ \Phi=f_{F}$. For the restrictions $f_{{\rm C},F}$ and $f_{{\rm C},F^{\prime}}$ to the complementary sets of the tubular neighborhoods, a similar relation holds.
\end{enumerate}
\end{Thm}
\begin{proof}
Consider the restriction of $f$ to a small closed tubular neighborhood $N(F)$ of $F$. This gives a trivial linear $F$-bundle over the $4$-dimensional standard disk $D$, which is smoothly embedded in ${\mathbb{R}}^4$ of the target. Consider exchanging a bundle isomorphism to attach the closed tubular neighborhood to the complement $M-{\rm Int}\ N(F)$. The bundles on the boundaries are both regarded as trivial smooth and linear $S^{m-n}$-bundles over the $3$-dimensional standard sphere $\partial D$, which is smoothly embedded in ${\mathbb{R}}^4$. We may assume that a $4$-dimensional polyhedron ($4$-th cycle) by which ${\rm PD}([F])$ is represented and $F$ intersect in a finite subset of $M$ which is also in the interior of some $4$-dimensional simplex in the polyhedron. The polyhedron is locally a $4$-dimensional disk around each point of the finite subset. By the conditions on the order and the divisibility of $[F]$ and the last statement of Theorem \ref{thm:3} together with fundamental discussions on the obstruction theory of linear bundles, we consider a new suitable bundle isomorphism between the $S^{m-n}$-bundles over the $3$-dimensional standard sphere regarding the bundles as the trivial linear bundles.
As a result we naturally obtain a new map $f^{\prime}$ satisfying the second and the third properties. The last statement of Theorem \ref{thm:3} together with fundamental discussions on the obstruction theory of linear bundles also contribute to the proof of the first property. This completes the proof.
\end{proof}

In the following example, to construct fold maps to which we can apply Theorem \ref{thm:4} explicitly, we introduce a procedure of constructing a new fold map from a given fold map by a kind of surgery operations based on discussions in \cite{kitazawa5} and \cite{kitazawa6}. Also, \cite{kobayashi}, \cite{kobayashi2} and \cite{kobayashisaeki} are closely related to them and have motivated the author to present these preprints. This is also important in the following subsection.

Let $f:M \rightarrow {\mathbb{R}}^n$ be a fold map on an $m$-dimensional closed and connected manifold $M$ satisfying the relation $m>n \geq 1$ and $S$ be the one-point set or a standard sphere in the regular value set such that the projection
$f {\mid}_{f^{-1}(S)}:f^{-1}(S) \rightarrow S$ gives a trivial $S^{m-n}$-bundle over $S$. Consider a
small closed tubular neighborhood $N(S)$. We can construct a fold map $f^{\prime}:M^{\prime} \rightarrow {\mathbb{R}}^n$ on an $m$-dimensional closed and connected manifold $M^{\prime}$ satisfying the following properties.
\begin{enumerate}
\item $f^{\prime}(S(f^{\prime}))=f(S(f)) \sqcup \partial N(S)$.
\item For a small closed tubular neighborhood $N^{\prime}(S)$ of $S$ satisfying $N(S) \subset {\rm Int}\  N^{\prime}(S)$, there exists a
diffeomorphism $\Phi:f^{-1}({\mathbb{R}}^n-{\rm Int}\  N^{\prime}(S)) \rightarrow {f^{\prime}}^{-1}({\mathbb{R}}^n-{\rm Int}\  N^{\prime}(S))$ satisfying $f=f^{\prime} \circ \Phi$ on $f^{-1}({\mathbb{R}}^n-{\rm Int}\  N^{\prime}(S))$.
\item For a small closed tubular neighborhood $N^{\prime \prime}(S)$ of $S$ satisfying $N^{\prime \prime}(S) \subset {\rm Int}\  N(S)$, the composition of ${f^{\prime}} {\mid}_{{f^{\prime}}^{-1}(N^{\prime}(S)-{\rm Int}\  N^{\prime \prime}(S))}$ with a natural projection to $\partial N(S)$ gives a trivial smooth bundle: note that by the definition $N^{\prime}(S)-{\rm Int}\  N^{\prime \prime}(S)$ is regarded as a closed tubular neighborhood of $\partial N(S)$.
\item $f {\mid}_{f^{-1}(N^{\prime \prime}(S))}:f^{-1}(N^{\prime \prime}(S)) \rightarrow N^{\prime \prime}(S)$ gives a trivial smooth {\rm (}$S^{m-n} \sqcup S^{m-n}${\rm )}-bundle over $N^{\prime \prime}(S)$.
\end{enumerate}
We call this procedure of constructing $f^{\prime}$ from $f$ a {\it trivial bubbling operation with standard spheres}. The standard sphere or the one-point set $S$ is called the {\it generating manifold} of the operations. Later FIGURE \ref{fig:2} will present this operation by the image of a map with same notation (we do not use $N^{\prime \prime}(S)$ and so on in FIGURE \ref{fig:2}).
\begin{Ex}
\label{ex:3}
For example, let $m>n \geq 1$ be integers and let $f:M \rightarrow {\mathbb{R}}^n$ be a special generic map on an $m$-dimensional homotopy sphere $M$ into ${\mathbb{R}}^n$ whose restriction to the singular set is an embedding and whose singular set is diffeomorphic to $S^{n-1}$ (a two-point set in the case $n=1$). Let $S$ be a one-point set in the interior of the image. Then by a trivial bubbling operation with standard spheres whose generating manifold is a point, we obtain a new fold map $f^{\prime}:M^{\prime} \rightarrow {\mathbb{R}}^n$: in the case $n \geq 2$, the resulting map is a componentwisely trivial round fold map presented in Theorem \ref{thm:2} whose singular set consists of two connected components.
In the cases $n=4$, note that in considerable cases, we have examples explaining Theorem \ref{thm:4}.
\end{Ex}
\subsection{New examples of fold maps on $7$-dimensional manifolds and the topologies of these $7$-dimensional manifolds.}
\begin{Lem}
\label{lem:2}
Let $R$ be a commutative ring. For a special generic map $f:M \rightarrow {\mathbb{R}}^n$ with $m>n \geq 1$, two integers $0 \leq j_1,j_2 \leq m-n$ satisfying $j_1+j_2 \geq n$ and for classes $c_i \in H^{j_i}(M;R)$ {\rm (}i=1,2{\rm )}, the cup product of $c_1$ and $c_2$ vanishes.
\end{Lem}
\begin{proof}
By virtue of Proposition \ref{prop:2}, $M$ is bounded by a compact PL manifold $W$ simple homotopy equivalent to $W_f$, which is an $n$-dimensional compact manifold and has the homotopy type of an ($n-1$)-dimensional polyhedron.
We apply Proposition \ref{prop:3} to complete the proof. Let $i:M \rightarrow W$ denote the inclusion. For the classes $c_1 \in H^{j_1}(M;R)$ and $c_2 \in H^{j_2}(M;R)$ satisfying $j_1+j_2 \geq n$, by Proposition \ref{prop:3}, the product of ${i^{\ast}}^{-1}(c_1)$ and ${i^{\ast}}^{-1}(c_2)$, which are well-defined, vanishes. If we restrict this to $M$, then we obtain the product of $c_1$ and $c_2$ and this vanishes.
\end{proof}
The following proposition and its proof play important roles in Theorem \ref{thm:5}.
\begin{Prop}
\label{prop:6}
Let $m>n \geq 1$ be integers and $M$ be an $m$-dimensional closed and connected manifold. We also assume the relation $0<n-\dim S<m-n<n<m-n+\dim S<m$. Let $S$ be a standard sphere or a point smoothly embedded in ${\mathbb{R}}^n$ with a trivial normal bundle.
For a smooth map $f:M \rightarrow {\mathbb{R}}^n$ to which we can do a trivial bubbling operation with standard spheres whose generating manifold is $S$, we can do such an operation so that the resulting new map $f^{\prime}:M^{\prime} \rightarrow {\mathbb{R}}^n$ satisfies the following three properties.
\begin{enumerate}
\item $H_{j}(M^{\prime};\mathbb{Z})$ is isomorphic to $H_{j}(M;\mathbb{Z}) \oplus \mathbb{Z}$ for $j=n-\dim S,m-n,n,m-n+\dim S$.
\item $H_{j}(M^{\prime};\mathbb{Z})$ is isomorphic to $H_{j}(M;\mathbb{Z})$ for $j \neq n-\dim S,m-n,n,m-n+\dim S$.
\item The $j$-th Stiefel-Whitney class of $M^{\prime}$ vanishes if that of $M$ vanishes for any $j$.
\end{enumerate}
\end{Prop}
\begin{proof}
We represent the change of the image and the singular value set of $f$ by a trivial bubbling operation with standard spheres whose generating manifold is $S$ in FIGURE \ref{fig:2} and FIGURE \ref{fig:3}. We discuss the topologies of some manifolds and maps on them related to the proof.

$N(S)$ is a small closed tubular neighborhood of $S$ and regarded as the total space of a trivial linear $D^{n-\dim S}$-bundle by the assumption on the normal bundle. $f {\mid}_{f^{-1}(N(S))}$ gives a trivial smooth $S^{m-n}$-bundle over $N(S)$ and regarded as the product map of the projection of a trivial smooth $S^{m-n}$-bundle over $D^{n-\dim S}$ and the identity map ${\rm id}_{S}$. We can construct $f^{\prime}$ so that $f^{\prime} {\mid}_{{f^{\prime}}^{-1}(N(S))}$ is regarded as the product map of a smooth map as presented in FIGURE \ref{fig:3} and the identity map ${\rm id}_{S}$.

$D$ is a fiber of a small closed tubular neighborhood $N^{\prime}(S)$ of $S$ satisfying $N(S) \subset {\rm Int}\  N^{\prime}(S)$ as presented in the definition of a trivial bubbling operation with standard spheres: $N^{\prime}(S)$ is also regarded as a trivial linear bundle over $S$. We explain about the structure of the map in FIGURE \ref{fig:3}. This is obtained as a restriction of a round fold map on $S^{n-\dim S} \times S^{m-n}$ into ${\mathbb{R}}^{n-\dim S}$ in Theorem \ref{thm:2} or a map as in
Example \ref{ex:3} whose singular set consists of two connected components to a compact submanifold obtained by removing the interior of a small closed tubular neighborhood of $S^{n-\dim S-1} \times \{\ast\} \subset S^{n-\dim S} \times S^{m-n}$ for a point $\ast \in S^{m-n}$ (for $n-\dim S \geq 2$ and we can naturally obtain a similar function for $n-\dim S=1$: see \cite{kitazawa4}) where $S^{n-\dim S-1}$ is an equator of $S^{n-\dim S}$. For the construction of the original round fold map, see \cite{kitazawa} and \cite{kitazawa2} and see also \cite{kitazawa4}.

By the observation of the topologies of ${f^{\prime}}^{-1}(N(S))$ and $f^{-1}(N(S))$, we can construct $f^{\prime}$ so that the resulting manifold $M^{\prime}$ satisfies the first two properties by taking the bundle isomorphism between the trivial
$S^{m-n}$-bundles over $\partial N(S)$ on the boundaries for the identification in a suitable method.
We explain about the summand $\mathbb{Z}$ of $H_{j}(M;\mathbb{Z}) \oplus \mathbb{Z}$, isomorphic to $H_{j}(M^{\prime};\mathbb{Z})$.
For $j=n-\dim S$, it is generated by a class represented by $S^{n-\dim S} \times \{{\ast}^{\prime}\} \subset S^{n-\dim S} \times S^{m-n}$ in the domain of the map in FIGURE \ref{fig:3}: we must take a point ${\ast}^{\prime}$ sufficiently far from $\ast$. For $j=m-n$, it is generated by a class represented by $\{{\ast}^{\prime \prime}\} \times S^{m-n} \subset S^{n-\dim S} \times S^{m-n}$: we must take ${\ast}^{\prime \prime}$ outside the equator $S^{n-\dim S-1}$. For $j=n,m-n+\dim S$, take the products of the image of a suitable section of the new trivial bundle $f {\mid}_{{f^{\prime}}^{-1}(S)}:{f^{\prime}}^{-1}(S) \rightarrow S$ and the submanifolds by which the generators of the subgroups of the ($n-\dim S$)-th and ($m-n$)-th homology groups of $M^{\prime}$ before are represented. Desired classes are represented by these products.

We can see the third condition easily from the construction.
This completes the proof.
\end{proof}
\begin{figure}
\includegraphics[width=50mm]{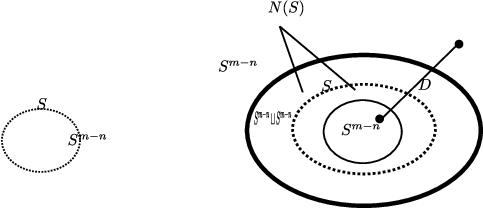}
\caption{The change of the images and the singular value sets of a smooth map by a trivial bubbling operation with spheres. $S$ is the generating manifold, $N(S)$ is a closed tubular neighborhood and $\partial N(S)$ is the new connected component of the singular value set of $f^{\prime}$: $S^{m-n}$ and $S^{m-n} \sqcup S^{m-n}$ are for preimages of regular values as FIGURE \ref{fig:1}. $D$ is a fiber of $N^{\prime}(S)$, diffeomorphic to $D^{n-\dim S}$, in the definition of a trivial bubbling operation with standard spheres: it is regarded as a trivial bundle over $S$ and satisfies $N(S) \subset {\rm Int}\  N^{\prime}(S)$.}
\label{fig:2}
\end{figure}
\begin{figure}
\includegraphics[width=50mm]{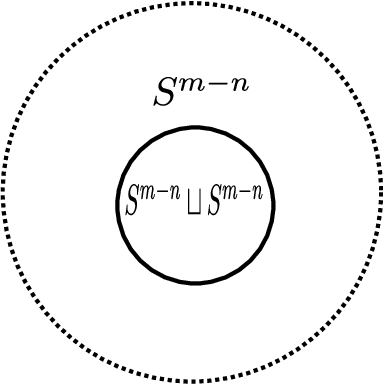}
\caption{The image and the singular value of the map $f^{\prime} {\mid}_{{f^{\prime}}^{-1}(D)}$ onto $D$ explained in FIGURE \ref{fig:2}: $S^{m-n} $and $S^{m-n} \sqcup S^{m-n}$ are for preimages of the regular values and the dotted line represents the boundary of the image (on the preimage of this, the map gives a trivial $S^{m-n}$-bundle over the boundary).}
\label{fig:3}
\end{figure}
The following theorem is another new result of the present paper. We apply Proposition \ref{prop:6} for $(m,n,\dim S)=(7,4,2)$.

For a finitely generated free commutative group $A$ and a basis $\{a_j\}$, we define the {\it dual} ${a_j}^{\ast}$ of $a_j$ as a homomorphism into $\mathbb{Z}$ satisfying ${a_{j_1}}^{\ast}(a_{j_2})={\delta}_{j_1,j_2}$ where ${\delta}_{j_1.j_2}$ is the Kronecker's delta
(${\delta}_{j_1.j_2}=1$ if $j_1 = j_2$ and otherwise $0$).
Moreover, if the group is a homology group, then we can regard the dual as a cohomology class in a canonical way and we regard the dual as the cohomology class obtained in this way.

\begin{Thm}
\label{thm:5}
Let $A$, $B$ and $C$ be free finitely generated commutative groups of rank $a$,$b$ and $c$, respectively. Let $\{a_j\}_{j=1}^{a}$, $\{b_j\}_{j=1}^{b}$ and $\{c_j\}_{j=1}^{c}$ be bases of $A$, $B$ and $C$, respectively.
For each integer $1 \leq i \leq b$, let $\{a_{i,j}\}_{j=1}^{a}$ be a sequence of integers.
Let $p \in B \oplus C$.
In this situation, there exist a $7$-dimensional closed and simply-connected spin manifold $M$ and a fold map $f:M \rightarrow {\mathbb{R}^4}$ satisfying the following properties.
\begin{enumerate}
\item $H_2(M;\mathbb{Z})$ is isomorphic to $A \oplus B$ and $H_4(M;\mathbb{Z})$ is isomorphic to $B \oplus C$. $H_3(M;\mathbb{Z})$ is free.
\item By universal coefficient theorem, $H^j(M;\mathbb{Z})$ is isomorphic to $H_j(M;\mathbb{Z})$. Via suitable isomorphisms ${\phi}_2:A \oplus B \rightarrow H_2(M;\mathbb{Z})$ and ${\phi}_4:B \oplus C \rightarrow H_4(M;\mathbb{Z})$, we can identify
$H_2(M;\mathbb{Z})$ with $A \oplus B$ and $H_4(M;\mathbb{Z})$ with $B \oplus C$. Furthermore, we can define the duals ${a_j}^{\ast}$, ${b_{j,2}}^{\ast}$, ${b_{j,4}}^{\ast}$ and ${c_j}^{\ast}$ of ${\phi}_2((a_j,0))$, ${\phi}_2((0,b_j))$, ${\phi}_4((b_j,0))$ and
${\phi}_4((0,c_j))$, respectively, and we can canonically obtain bases of $H^2(M;\mathbb{Z})$ and $H^4(M;\mathbb{Z})$. Let the composition of ${\phi}_4$ with an isomorphism mapping the elements of the basis to their Poincar\'e duals be denoted by
${\phi}_{4,3}:B \oplus C \rightarrow H^3(M;\mathbb{Z})$. Let the composition of ${\phi}_2$ with an isomorphism mapping the elements of the basis to their Poincar\'e duals be denoted by ${\phi}_{2,5}:A \oplus B \rightarrow H^5(M;\mathbb{Z})$.
For the products in $H^i(M;\mathbb{Z})$ for $i=4,5,7$, the following properties hold.
\begin{enumerate}
\item Products are in $H^4(M;\mathbb{Z})$.
\begin{enumerate}
\item The product of ${a_{j_1}}^{\ast}$ and ${a_{j_2}}^{\ast}$ and that of ${b_{j_1,2}}^{\ast}$ and ${b_{j_2,2}}^{\ast}$ vanish for any $j_1$ and $j_2$.
\item The product of ${a_{j_1}}^{\ast}$ and ${b_{j_2,2}}^{\ast}$ is $a_{j_2,j_1}{b_{j_2,4}}^{\ast} \in H^4(M;\mathbb{Z})$. for any $j_1$ and $j_2$.
\end{enumerate}
\item Products are in $H^5(M;\mathbb{Z})$.
\begin{enumerate}
\item The product of ${a_j}^{\ast}$ and ${\phi}_{4,3}((0,c))$ vanishes for any $j$ and $c \in C$.
\item The product of ${b_{j,2}}^{\ast}$ and ${\phi}_{4,3}((0,c))$ vanishes for any $j$ and $c \in C$.
\item The product of ${b_{j_1,2}}^{\ast}$ and ${\phi}_{4,3}((b_{j_2},0))$ vanishes for any pair $(j_1,j_2)$ of distinct numbers.
\item The product of ${a_{j_1}}^{\ast}$ and ${\phi}_{4,3}((b_{j_2},0))$ is $a_{j_2,j_1}{\phi}_{2,5}((0,b_{j_2})) \in H^5(M;\mathbb{Z})$. for any $j_1$ and $j_2$.
\item The product of ${b_{i,2}}^{\ast}$ and ${\phi}_{4,3}((b_{i},0))$ is ${\Sigma}_{j=1}^{a} a_{i,j}{\phi}_{2,5}((a_j,0)) \in H^5(M;\mathbb{Z})$.
\end{enumerate}
\item Products are in $H^7(M;\mathbb{Z})$.
\begin{enumerate}
\item The product of  ${a_j}^{\ast}$ and ${\phi}_{2,5}((a_j,0))$ forms a generator of $H^7(M;\mathbb{Z})$.
\item The product of  ${a_{j_1}}^{\ast}$ and ${\phi}_{2,5}((a_{j_2},0))$ vanishes for any pair $(j_1,j_2)$ of distinct numbers.
\item The product of  ${a_j}^{\ast}$ and ${\phi}_{2,5}((0,b))$ vanishes for any $j$ and $b \in B$.
\item The product of  ${b_{j,2}}^{\ast}$ and ${\phi}_{2,5}((a,0))$ vanishes for any $j$ and $a \in A$.
\item The product of  ${b_{j,2}}^{\ast}$ and ${\phi}_{2,5}((0,b_j))$ forms a generator of $H^7(M;\mathbb{Z})$.
\item The product of  ${b_{j_1,2}}^{\ast}$ and ${\phi}_{2,5}((0,b_{j_2}))$ vanishes for any pair $(j_1,j_2)$ of distinct numbers.
\item The product of ${b_{j,4}}^{\ast}$ and ${\phi}_{4,3}((b_j,0))$ forms a generator of $H^7(M;\mathbb{Z})$.
\item The product of  ${b_{j_1,4}}^{\ast}$ and ${\phi}_{4,3}((b_{j_2},0))$ vanishes for any pair $(j_1,j_2)$ of distinct numbers.
\item The product of  ${b_{j,4}}^{\ast}$ and ${\phi}_{4,3}((0,c))$ vanishes for any $j$ and $c \in C$.
\item The product of  ${c_j}^{\ast}$ and ${\phi}_{4,3}((b,0))$ vanishes for any $j$ and $b \in B$.
\item The product of ${c_j}^{\ast}$ and ${\phi}_{4,3}((0,c_j))$ forms a generator of $H^7(M;\mathbb{Z})$.
\item The product of ${c_{j_1}}^{\ast}$ and ${\phi}_{4,3}((0,c_{j_2}))$ vanishes for any pair $(j_1,j_2)$ of distinct numbers.
\end{enumerate}
\end{enumerate}
\item The 3rd and the 5th Stiefel-Whitney classes of $M$ vanish.
\item Let $d$ be the isomorphism $d:H_4(M;\mathbb{Z}) \rightarrow H^4(M;\mathbb{Z})$ defined by corresponding the duals. The first Pontryagin class of $M$ is regarded as $4d \circ {\phi}_4(p) \in H^4(M;\mathbb{Z})$. The 4th Stiefel-Whitney class of $M$ vanishes.
\item $f {\mid}_{S(f)}$ is an embedding.
\item The index of each singular point is $0$ or $1$.
\item Preimages of regular values in the image of $f$ are diffeomorphic to $S^3$ or $S^3 \sqcup S^3$.
\end{enumerate}
\end{Thm}
\begin{proof}
We consider a special generic map $f_0$ as in Example \ref{ex:1} into ${\mathbb{R}}^4$ on a manifold represented as a connected sum of $a \geq 0$ copies of $S^2 \times S^5$.
We can identify $a_j \in A$ in the assumption with a generator of the 2nd homology group of $j$-th copy of $S^2 \times S^5$. This is important in constructing ${\phi}_2:A \oplus B \rightarrow H_2(M;\mathbb{Z})$: we abuse this notation for elements in $A$ and the 2nd homology classes both in several situations.
The image is represented as a boundary connected sum of $a$ copies of $S^2 \times D^2$. Note that the manifold is spin: we will apply Proposition \ref{prop:6} one after another and this is a main ingredient.

First we can choose a family consisting of $2$-dimensional standard spheres and points embedded disjointly and smoothly in the interior of the image of $f_0$ so that the following properties hold.
\begin{enumerate}
\item The number of the $2$-dimensional standard spheres is $b \geq 0$ and the number of the points is $c \geq 0$. Let $\{S_j\}_{j=1}^b$ denote the family of the $2$-dimensional standard spheres.
\item For the image, represented as a boundary connected sum of $a$ copies of $D^2 \times S^2$, let $a_j$ denote a generator of the 2nd homology group of the $j$-th copy of $D^2 \times S^2$. Here we identify the generator with $a_j \in A$ of the assumption. ${\Sigma}_{j=1}^{a} a_{i,j}a_j \in H_2(W_{f_0};\mathbb{Z})$ is represented by $S_j$.
\end{enumerate}

Note that for the second property here, the relation $2 \times 2=4$ is essential. This allows us to construct smooth embeddings sufficiently freely.
More precisely, we first consider the case where $a_j \in H_2(W_{f_0};\mathbb{Z})$ is represented by $S_{i}$ and after that consider a connected sum of them as smoothly embedded submanifolds with orientations.
We can do trivial bubbling operations with standard spheres one after another so that each chosen submanifold is the generating manifold (in the image of each step) for each operation. We can do the operations so that for the resulting map and the $7$-dimensional manifold, the fifth, sixth and seventh properties hold.

We prove the first property and statements on 2nd and 4th homology groups of the second property. To quote some of the description in the proof of Proposition \ref{prop:6} in the proof of the present theorem, we set $(m,n)=(7.4)$. We replace the submanifold $S$ in (the proof of) Proposition \ref{prop:6} by a suitable sphere $S_i$ in the present theorem.

We can identify the homology class represented by $S^{n-\dim S_i} \times \{{\ast}^{\prime}\} \subset S^{n-\dim S_i} \times S^{m-n}$ in the proof of Proposition \ref{prop:6} with $b_i \in B$ in the assumption as we have done for $a_j$'s in the beginning and as a result we obtain a desired isomorphism ${\phi}_2:A \oplus B \rightarrow H_2(M;\mathbb{Z})$. Consider the product of the image of a suitable section of a new trivial smooth bundle $f {\mid}_{{f}^{-1}(S_i)}:{f}^{-1}(S_i) \rightarrow S_i$ and the submanifold $S^{n-\dim S_i} \times \{{\ast}^{\prime}\} \subset S^{n-\dim S_i} \times S^{m-n}$ in the proof of Proposition \ref{prop:6}. For the $c$ points $p_i$, consider the submanifold $S^{n-\dim \{p_i\}} \times \{{\ast}^{\prime}\} \subset S^{n-\dim \{p_i\}} \times S^{m-n}$. This discussion yields a desired isomorphism ${\phi}_4:B \oplus C \rightarrow H_4(M;\mathbb{Z})$.

We prove properties on products of cohomology classes in the second property of the seven properties.

We prove for cases where products are in $H^4(M;\mathbb{Z})$.

Consider the product of a cocycle representing the dual ${b_{i,2}}^{\ast}$ of the class $b_{i}$ and a cocycle representing the dual ${a_{j}}^{\ast}$ of the class $a_{j}$ and value at the cycle represented as the tensor product of a cycle in the class
${\Sigma}_{j=1}^{a} a_{i,j}a_j$ and a cycle in the class represented by $S^{n-\dim S_i} \times \{{\ast}^{\prime}\} \subset S^{n-\dim S_i} \times S^{m-n}$ before (we consider a canonically obtained cohomology class in the product $M \times M$ and the product is obtained as the pull-back via the diagonal map from $M$ into $M \times M$).

We prove for cases where products are in $H^5(M;\mathbb{Z})$.
We can see that ${\phi}_{4,3} \circ {{\phi}_{4}}^{-1}$ maps each homology class before to the dual of the homology class represented by a connected component $\{{\ast}^{\prime \prime}\} \times S^{m-n}$
of the preimage of a suitable regular value. We explain about ${\phi}_{2,5} \circ {{\phi}_2}^{-1}$. This maps each ($n-\dim S_i$)-th homology class represented by $S^{n-\dim S_i} \times \{{\ast}^{\prime}\} \subset S^{n-\dim S_i} \times S^{m-n}$ to the dual of the homology class represented by the product of a fiber of the smooth trivial bundle $S^{n-\dim S_i} \times S^{m-n}$ in the proof of Proposition \ref{prop:6} over $S^{n-\dim S_i}$ and the image of a section of the trivial smooth bundle $f {\mid}_{f^{-1}(S_i)}$ over $S_i$. We must not forget that the domain of the original special generic map $f_0$ is represented as a connected sum of finitely many copies of $S^2 \times S^5$ and that the class to which this isomorphism maps $a_j \in H_2(M;\mathbb{Z})$ is the class whose dual is ${\phi}_{2,5} \circ {{\phi}_2}^{-1}(a_j) \in H^5(M;\mathbb{Z})$ in the proof.

The first three products vanish. In fact, for the pair of the cohomology classes we consider, we can take corresponding cycles, representing the classes whose duals are these cohomology classes so that each class does not vanish at the corresponding cycle and that they do not intersect. We can argue as this by observing trivial bubbling operations with standard spheres and the first property on homology groups.

We prove the first case of the last two products where the products are in $H^5(M;\mathbb{Z})$. As one of keys, remember that the class  represented by $S_{j_2}$ is represented as ${\Sigma}_{j_1=1}^{a} a_{j_2,j_1}a_{j_1} \in H_2(W_{f_0};\mathbb{Z})$ in the proof in cases where products are in $H^4(M;\mathbb{Z})$. The other key is in the argument in the proof of Proposition \ref{prop:6} just before: ${\phi}_{2,5}$ maps each ($n-\dim S_{j_2}$)-th homology class represented by $S^{n-\dim S_{j_2}} \times \{{\ast}^{\prime}\} \subset S^{n-\dim S_{j_2}} \times S^{m-n}$ to the dual of the homology class represented by the product of the image of a section of the trivial smooth bundle $f {\mid}_{f^{-1}(S_{j_2})}$ over $S_{j_2}$ and a fiber of the trvial smooth bundle $S^{n-\dim S_{j_2}} \times S^{m-n}$ over $S^{n-\dim S_{j_2}}$. We can see that the product of ${a_{j_1}}^{\ast}$ and ${\phi}_{4,3}((b_{j_2},0))$ is $a_{j_2,j_1}{\phi}_{2,5}((0,b_{j_2})) \in H^5(M;\mathbb{Z})$. for any $j_1$ and $j_2$.

We prove the last case. The coefficient of $a_j$ in the class ${\Sigma}_{j=1}^{a} a_{i,j}a_j \in H_2(W_{f_0};\mathbb{Z})$, which is represented by the sphere $S_{i}$, is $a_{i,j}$, By virtue of the discussion in the proof of Proposition \ref{prop:6} just before and by considering the class to which the isomorphism ${\phi}_{2,5} \circ {{\phi}_2}^{-1}$ maps
$a_j \in H_2(M;\mathbb{Z})$, the value of the product at the presented class is $a_{i,j}$. For the product of ${b_{i,2}}^{\ast}$ and ${\phi}_{4,3}((b_{i},0))$, consider the value at the class whose dual is
${\phi}_{2,5} \circ {{\phi}_2}^{-1}(a_j) \in H^5(M;\mathbb{Z})$: it is $a_{i,j}$. This completes the proof of the last case and the proof for cases where products are in $H^5(M;\mathbb{Z})$.

We explain about cases where products are in $H^7(M;\mathbb{Z})$. For cases where products vanish, we can argue as in similar cases for cases of $H^5(M;\mathbb{Z})$: for the pair of the cohomology classes we consider, we can take corresponding cycles, representing the classes whose duals are these cohomology classes so that each class does not vanish at the corresponding cycle and that they do not intersect. Remaining cases can be shown by virtue of the definitions of ${\phi}_{2,5}$ and ${\phi}_{4,3}$ and Poincare's duality theorem immediately.

For the third property, the 3rd Stiefel-Whitney class vanishes since homology groups are free. For the 5th Stiefel-Whitney class, first review the proof of Proposition \ref{prop:6} or the second property for cases where products are in
$H^5(M;\mathbb{Z})$. Take a generator of the ($m-n+\dim S$)-th homology group of the resulting manifold ($m-n=3$ and
$\dim S=2$ here) in the proof of Proposition \ref{prop:6}. The generator is regarded as a class represented by a product of standard spheres and the normal bundle of this product which is also a smooth submanifold is $2$-dimensional and trivial. In this argument we must not forget that the domain of the original special generic map $f_0$ is represented as a connected sum of a finite copies of $S^2 \times S^5$ taken in the smooth category and spin. This completes the proof of the fact that the 5th Stiefel-Whitney class vanishes.
We have the fourth property by virtue of Theorem \ref{thm:4} and a discussion in the proof of Proposition \ref{prop:6} for the $n$-th homology class where $n=4$: consider a connected component of a preimage consisting of two copies of $S^3$ of a regular value and apply Theorem \ref{thm:4}. On the 4th Stiefel-Whitney class, discussed in the fourth property, we can see the vanishment by fundamental properties of Stiefel-Whitney classes and Euler classes together with some ingredients of Theorem \ref{thm:3}.

This completes the proof.
\end{proof}

Note for example that by virtue of this theorem we can know the cohomology ring of the manifold $M$ completely.
On the discussion on cohomology rings, see also \cite{kitazawa6} for example. \cite{kitazawa6} is, as in Remark \ref{rem:2}, mainly devoted to studies of polyhedra representing the manifolds compactly and the present theorem and its proof is a new work.




\begin{Cor}
\label{cor:2}
In the situation of Theorem \ref{thm:5}, if at least one of the following hold, then $M$ admits no special generic maps into
${\mathbb{R}}^4$.
\begin{enumerate}
\item $p \in B \oplus C$ is not zero.
\item In at least one sequence $\{a_{i,j}\}_{j=1}^{a}$, at least one non-zero number exists.
\end{enumerate}
\end{Cor}
\begin{proof}
By virtue of Lemmas \ref{lem:1} and \ref{lem:2}, we have the result.
\end{proof}
For the family of $7$-dimensional manifolds we can obtain in Theorem \ref{thm:5} with Corollary \ref{cor:2}, their homology groups and fundamental groups are isomorphic to those of some $7$-dimensional manifolds admitting special generic maps in Example \ref{ex:1} into ${\mathbb{R}}^4$. Moreover, Propositions \ref{prop:3} and \ref{prop:5} state that homology groups of $7$-dimensional closed and simply-connected manifolds admitting special generic maps into ${\mathbb{R}}^4$ are isomorphic to these groups.

Theorem \ref{thm:1} and Corollary \ref{cor:2} may imply that special generic maps restrict the cohomology rings and the classical characteristic classes such as Stiefel-Whitney classes and Pontragin classes of $7$-dimensional closed and (simply-)connected manifolds considerably.

On the other hand, Corollary \ref{cor:1} and Theorem \ref{thm:5} may imply the following fact: the class of fold maps such that the restrictions to the singular sets are embeddings and that preimages of regular values are disjoint unions of spheres on $7$-dimensional closed manifolds (into ${\mathbb{R}}^4$) or suitable wider classes cover considerably wide classes of $7$-dimensional closed and (simply-)connected manifolds, as manifolds admitting these maps.

Note that recently Kreck (\cite{kreck}) studied $7$-dimensional closed and simply-connected manifolds with free 2nd homology groups and that he found a comprehensive classification in terms of explicit (co)bordism relations. Theorem \ref{thm:5} may be a pioneering result for us to understand the class of these $7$-dimensional manifolds in more geometric and constructive ways.

Note also that before his study, comprehensive algebraic topological and differential topological classifications of $7$-dimensional closed and $2$-connected manifolds were studied in \cite{crowleyescher} and \cite{crowleynordstrom} for example.

We will show another version of Theorem \ref{thm:5} for $7$-dimensional manifolds which may not be spin.

The following gives special generic maps into ${\mathbb{R}}^n$ resembling ones in Example \ref{ex:1}. See \cite{saeki} and in this case see also \cite{nishioka} for example.

\begin{Ex}
\label{ex:5}
Let $m>4$, $n \geq 3$ and $l>0$ be integers. Set $a_j=0,1$ for each integer $1 \leq j \leq l$.

In this situation, there exists an $m$-dimensional closed and connected manifold $M$ represented as a connected sum of $l>0$ total spaces of linear $S^{m-2}$-bundles over $S^2$ taken in the smooth category and a special generic map
$f:M \rightarrow {\mathbb{R}}^n$ such that the following properties hold.
\begin{enumerate}
\item $W_f$ is represented as a boundary connected sum of manifolds of exactly $l$ copies of $S^2 \times D^{n-2}$ taken in the smooth category.
\item $f {\mid}_{S(f)}$ is an embedding.
\item  For $W_f$, $H^2(W_f;\mathbb{Z}/2\mathbb{Z})$ is generated by the set of all 2nd cohomology classes represented as the duals of the classes represented by $S^2 \times \{\ast\} \subset S^2 \times {\rm Int}\  D^{n-2} \subset S^2 \times D^{n-2}$ in the family $\{S^{2} \times D^{n-2} \subset W_f\}$ of the $l$ copies. Let $\{{{\nu}_j}^{\ast}\}_{j=1}^l \subset H^2(M;\mathbb{Z})$ be the sequence of the cohomology classes obtained as the pull-backs of these 2nd cohomology classes represented as the duals of the homology classes ${\nu}_j$ represented by the copies of $S^2 \times \{\ast\} \subset S^2 \times D^{n-2}$. In this situation, the 2nd Stiefel-Whitney class of $M$ is ${\Sigma}_{j=1}^{l} a_j{{\nu}_j}^{\ast} \in H^2(M;\mathbb{Z}/2\mathbb{Z})$ and the trivial smooth $S^{m-3}$-bundle $f {\mid}_{f^{-1}(S_{{\nu}_i})}$ over a $2$-dimensional sphere $S_{{\nu}_i}$ in ${\rm Int}\  W_f \subset W_f$ is linear and trivial if $a_i=0$ and not trivial if $a_i=1$ where the 2nd homology class ${\nu}_i$ is represented by $S_{{\nu}_i}$.
\end{enumerate}
\end{Ex}
We will prove Proposition \ref{prop:6} where $n-\dim S=2$ holds with a few additional statements.
\begin{Prop}
\label{prop:7}
Let $m>n \geq 1$ be integers and let $M$ be an $m$-dimensional closed and connected manifold. We also assume the relation $0<n-\dim S=2<m-n<n<m-n+\dim S=m-2$. Let $S$ be a standard sphere or a one-point set smoothly embedded in ${\mathbb{R}}^n$ with a trivial normal bundle. For a smooth map $f:M \rightarrow {\mathbb{R}}^n$ to which we can do a trivial bubbling operation with standard spheres whose generating manifold is $S$, we can do such an operation so that for the resulting new map $f^{\prime}:M^{\prime} \rightarrow {\mathbb{R}}^n$, the following properties hold.

\begin{enumerate}
\item $H_{j}(M^{\prime};\mathbb{Z})$ is isomorphic to $H_{j}(M;\mathbb{Z}) \oplus \mathbb{Z}$ for $j=2,m-n,n,m-n+\dim S$.
\item $H_{j}(M^{\prime};\mathbb{Z})$ is isomorphic to $H_{j}(M;\mathbb{Z})$ for $j \neq 2,m-n,n,m-n+\dim S$.
\item $H^{2}(M^{\prime};\mathbb{Z}/2\mathbb{Z})$ is isomorphic to $H^{2}(M;\mathbb{Z}/2\mathbb{Z}) \oplus \mathbb{Z}/2\mathbb{Z}$.
\end{enumerate}
Furthermore, for a suitable isomorphism $\phi$ from $H^{2}(M;\mathbb{Z}/2\mathbb{Z}) \oplus \mathbb{Z}/2\mathbb{Z}$ onto $H^{2}(M^{\prime};\mathbb{Z}/2\mathbb{Z})$, if the 2nd Stiefel-Whitney class of $M$ is $a_0$ and $b=0,1$, then we can construct $M^{\prime}$ so that the 2nd Stiefel-Whitney class is $\phi((a_0,b))$.
\end{Prop}
\begin{proof}[A sketch of the proof of only additional statements]
The third property and the additional statement on the 2nd Stiefel-Whitney classes of the manifolds in the last are additional statements. The third property is in fact straightforward. We can prove the property as the proof of Proposition \ref{prop:6}: by observing the topologies of ${f^{\prime}}^{-1}(N(S))$ and $f^{-1}(N(S))$ as in the proof of Proposition \ref{prop:6}, we can do ($N(S)$ is a closed tubular neighborhood as the proof of Proposition \ref{prop:6}). For the additional statement on the 2nd Stiefel-Whitney classes of the manifolds, the case $b=0$ is also shown similarly to Proposition \ref{prop:6} by observing the construction. We need to show for the case $b=1$.
We replace the map in FIGURE \ref{fig:3} by a map obtained as a restriction of a round fold map on the total space $E$ of a linear $S^{m-n}$-bundle over $S^2$ whose Stiefel-Whitney class does not vanish into ${\mathbb{R}}^2$ in Theorem \ref{thm:2} or a map as in Example \ref{ex:3} whose singular set consists of two connected components to a compact submanifold obtained by removing the interior of a small closed tubular neighborhood of $S^{n-\dim S-1} \times \{\ast\} \subset E$ for a point $\ast \in S^{m-n}$ where $S^{n-\dim S-1}$ is an equator of $S^{n-\dim S}$ and where we consider a trivial bundle over the equator obtained by restricting a trivial bundle obtained by a suitable trivialization over a hemisphere (, diffeomorphic to $D^{n-\dim S}$,) in $S^{n-\dim S}$. By the construction, this completes the proof for the case $b=1$.
\end{proof}

We can give a theorem similar to Theorem \ref{thm:5} for $7$-dimensional closed and simply-connected manifolds which may not be spin.

\begin{Thm}
\label{thm:6}
Let $A$, $B$ and $C$ be free finitely generated commutative groups whose ranks are $a$,$b$ and $c$, respectively. Let $\{a_j\}_{j=1}^{a}$, $\{b_j\}_{j=1}^{b}$ and $\{c_j\}_{j=1}^{c}$ be bases of $A$, $B$ and $C$, respectively.
For each integer $1 \leq i \leq b$, let $\{a_{i,j}\}_{j=1}^{a}$ be a sequence of integers.
Let $p \in B \oplus C$.
Let $A^{\prime}:=A \otimes \mathbb{Z}/2\mathbb{Z}$ and $B^{\prime}:=B \otimes \mathbb{Z}/2\mathbb{Z}$ and let $q_1 \in A^{\prime}$ and $q_2 \in B^{\prime}$.
In this situation, there exist a $7$-dimensional closed and simply-connected manifold $M$ and a fold map $f:M \rightarrow {\mathbb{R}^4}$ satisfying the following properties.
\begin{enumerate}
\item $H_2(M;\mathbb{Z})$ is isomorphic to $A \oplus B$ and $H_4(M;\mathbb{Z})$ is isomorphic to $B \oplus C$.
$H_3(M;\mathbb{Z})$ is free.
\item By universal coefficient theorem, $H^j(M;\mathbb{Z})$ is isomorphic to $H_j(M;\mathbb{Z})$. Via suitable isomorphisms ${\phi}_2:A \oplus B \rightarrow H_2(M;\mathbb{Z})$ and ${\phi}_4:B \oplus C \rightarrow H_4(M;\mathbb{Z})$, we can identify $H_2(M;\mathbb{Z})$ with $A \oplus B$ and $H_4(M;\mathbb{Z})$ with $B \oplus C$. Furthermore, we can define the duals ${a_j}^{\ast}$, ${b_{j,2}}^{\ast}$, ${b_{j,4}}^{\ast}$ and ${c_j}^{\ast}$ of ${\phi}_2((a_j,0))$,
${\phi}_2((0,b_j))$, ${\phi}_4((b_j,0))$ and ${\phi}_4((0,c_j))$, respectively and we can canonically obtain bases of
$H^2(M;\mathbb{Z})$ and $H^4(M;\mathbb{Z})$.
For the products, the following two hold.
\begin{enumerate}
\item The product of ${a_{j_1}}^{\ast}$ and ${a_{j_2}}^{\ast}$ and that of ${b_{j_1,2}}^{\ast}$ and ${b_{j_2,2}}^{\ast}$ vanish for any $j_1$ and $j_2$.
\item The product of ${a_{j_1}}^{\ast}$ and ${b_{j_2,2}}^{\ast}$ is $2a_{j_2,j_1}{b_{j_2,4}}^{\ast} \in H^4(M;\mathbb{Z})$ for any $j_1$ and $j_2$.
\end{enumerate}
\item Let $d_2$ be the isomorphism $d_2:H_2(M;\mathbb{Z}/2\mathbb{Z}) \rightarrow H^2(M;\mathbb{Z}/2\mathbb{Z})$ defined canonically by corresponding the duals as before under the additional condition that $\mathbb{Z}/2\mathbb{Z}$ is tensored to each of the commutative groups, the 2nd Stiefel-Whitney class of $M$ is regarded as
$d_2 \circ {\phi}_2((q_1,q_2)) \in H^2(M;\mathbb{Z}/2\mathbb{Z})$: we canonically define ${\phi}_2$ for the case where the coefficient ring is $\mathbb{Z}/2\mathbb{Z}$.
\item The 3rd and the 5th Stiefel-Whitney classes of $M$ vanish.
\item Let $d$ be the isomorphism $d:H_4(M;\mathbb{Z}) \rightarrow H^4(M;\mathbb{Z})$ defined by corresponding the duals. The first Pontryagin class of $M$ is regarded as  $4d \circ {\phi}_4(p) \in H^4(M;\mathbb{Z})$. The 4th Stiefel-Whitney class of $M$ vanishes.
\item $f {\mid}_{S(f)}$ is an embedding.
\item The index of each singular point is $0$ or $1$.
\item Preimages of regular values in the image of $f$ are diffeomorphic to $S^3$ or $S^3 \sqcup S^3$.
\end{enumerate}
\end{Thm}
\begin{proof}
We consider a special generic map $f_0$ as in Example \ref{ex:5} into ${\mathbb{R}}^4$ on a manifold represented as a connected sum of $a \geq 0$ total spaces of linear $S^5$-bundles over $S^2$. The image is represented as a boundary connected sum of $a$ copies of $S^2 \times D^2$. We will apply Proposition \ref{prop:7} one after another and this is a main ingredient.

First as the proof of Theorem \ref{thm:5}, we can choose a family consisting of $2$-dimensional standard spheres and points embedded disjointly in the interior of the image of $f_0$ so that the following two properties hold.
\begin{enumerate}
\item The number of the $2$-dimensional standard spheres is $b \geq 0$ and the number of the points is $c \geq 0$. Let $\{S_j\}_{j=1}^b$ denote the family of the $2$-dimensional standard spheres.
\item For the image, represented as a boundary connected sum of $a$ copies of $D^2 \times S^2$, let $a_j$ denote the class represented by a generator of the 2nd homology group of the $j$-th copy of $D^2 \times S^2$ by identifying the generator with $a_j \in A$ of the assumption as we did in the proof of Theorem \ref{thm:5}. ${\Sigma}_{j=1}^{a} 2a_{i,j}a_j \in H_2(W_{f_0};\mathbb{Z})$ is represented by the class $S_{i}$.
\end{enumerate}
We can do trivial bubbling operations with standard spheres one after another so that each chosen submanifold is the generating manifold (in the image of each step) for each operation. The coefficients $2$ in ${\Sigma}_{j=1}^{a} 2a_{i,j}a_j \in H_2(W_{f_0};\mathbb{Z})$ are essential to enable us to do trivial bubbling operations with standard spheres: the smooth $S^3$-bundles over $S_j$ become trivial.
By a discussion similar to that of the proof of Theorem \ref{thm:5}, we can do the operations so that the resulting map and the $7$-dimensional manifold satisfy all the properties in the statements.
\end{proof}
\begin{Rem}
\label{rem:1}
In Theorem \ref{thm:6}, for example, in the second property, we can argue similarly for other products as in the proof of the second property of Theorem \ref{thm:5} and we can obtain essentially same facts.
\end{Rem}

We will present another theorem.

The total space of a trivial smooth $S^3$-bundle over a closed interval $I$ is diffeomorphic to $S^3 \times I$. The $j$-th homology group vanishes for $j \neq 0,3$. The $0$-th homology group whose coefficient ring is $\mathbb{Z}$ is isomorphic to $\mathbb{Z}$. The $3$rd homology group whose coefficient ring is $\mathbb{Z}$ is also isomorphic to $\mathbb{Z}$ and generated by the class represented by $S^3 \times \{0\}$.
Consider a Morse function $f_{{{\mathbb{C}P}^2}^{\prime}}$ on a $4$-dimensional compact manifold obtained by removing the interior of the disjoint union of two smoothly and disjointly embedded $4$-dimensional standard disks froma copy of the complex projective plane ${\mathbb{C}P}^2$.
The boundary is diffeomorphic to the disjoint union of two copies of $S^3$. We have the function satisfying the following properties.
\begin{enumerate}
\item It has exactly one singular point.
\item The preimage of the minimum and that of the maximum are connected components of the boundary.
\item Neither the minimum or the maximum is the singular value.
\end{enumerate}
The fundamental group and the $j$-th homology group vanishes for $j \neq 0, 2, 3$. The $0$-th homology group with coefficient ring $\mathbb{Z}$ is isomorphic to $\mathbb{Z}$. The 3rd homology group with coefficient ring $\mathbb{Z}$ is isomorphic to $\mathbb{Z}$ and generated by the class represented by a connected component of the boundary. The 2nd homology group with coefficient ring $\mathbb{Z}$ is isomorphic to $\mathbb{Z}$ and generated by the class represented by a $2$-dimensional sphere smoothly embedded in the interior of the $4$-dimensional manifold. Note that the manifold is not spin. Note also that this is regarded as a restriction of a Morse function $f_{{\mathbb{C}P}^2}$ with exactly three singular points on the complex projective plane. Moreover, the square of the dual of the class represented by the $2$-dimensional sphere in the original copy of the complex projective plane ${\mathbb{C}P}^2$ is a generator of the $4$-th integral cohomology group of the copyof ${\mathbb{C}P}^2$, isomorphic to $\mathbb{Z}$.

We consider a special generic map $f_0$ as in Example \ref{ex:1} into ${\mathbb{R}}^4$ on a manifold $M_0$ represented as a connected sum of $a \geq 0$ copies of $S^3 \times S^4$ considered in the smooth category. The image is represented as a boundary connected sum of $a$ copies of $S^3 \times D^1$ considered in the smooth category.

Starting from $f_0:M_0 \rightarrow {\mathbb{R}^4}$, we consider the following operation to construct a fold map
$f_{i+1}:M_{i+1} \rightarrow {\mathbb{R}}^4$ on a $7$-dimensional manifold $M_{i+1}$ from a fold map
$f_{i}:M_{i} \rightarrow {\mathbb{R}}^4$ on a $7$-dimensional manifold $M_{i}$ for each $i \geq 0$ one after another.

\begin{enumerate}
\item First we choose a $3$-dimensional standard sphere $S_i$ smoothly embedded in a connected component of the regular value set of $f_{i}$ so that the preimage is not empty and that
the preimage of the bounded connected component of ${\mathbb{R}}^4-S_i$ contains no singular point of index which is greater than $0$.
\item Consider a small closed tubular neighborhood $N(S_i)$ of $S_i$.
\item ${f_i} {\mid}_{{f_i}^{-1}(N(S_i))}:{f_i}^{-1}(N(S_i)) \rightarrow N(S_i)$ gives a trivial smooth $S^3$-bundle and replace this by the product map of the Morse function $f_{{{\mathbb{C}P}^2}^{\prime}}$ with the manifold of the target restricted to the image (as a result this function is surjective) and the identity map on a $3$-dimensional standard sphere to obtain a new fold map $f_{i+1}$ so that the restriction to the singular set is an embedding and that the number of connected components of the singular value set increases by one.
\end{enumerate}
By the definitions and properties of $f_{{{\mathbb{C}P}^2}^{\prime}}$ and the operation for example,
we can see that we can do the operations so that the following properties hold.
\begin{enumerate}
\item $M_{i+1}$ is simply-connected.
\item $H_2(M_{i+1};\mathbb{Z})$ is isomorphic to the direct sum of $H_2(M_i;\mathbb{Z}) \oplus \mathbb{Z}$ and we can set an isomorphism ${\phi}_{2,i}$ from the direct sum onto $H_2(M_{i+1};\mathbb{Z})$.
\item For the image of the original map $f_0$, represented as a boundary connected sum of $a$ copies of $S^3 \times D^1$, let $a_j$ denote a generator of the 3rd homology group of the $j$-th copy of $S^3 \times D^1$. $a_j$ can be represented by $3$-dimensional spheres $S_{j,i} \subset (S^3 \times {\rm Int}\  D^1) \bigcap ({\mathbb{R}}^4-f_{i}(S(f_i)))$ and $S_{j,i+1} \subset (S^3 \times {\rm Int}\  D^1) \bigcap ({\mathbb{R}}^4-f_{i+1}(S(f_{i+1})))$ which are smoothly and disjointly embedded and consider the preimages ${f_{i}}^{-1}(S_{j,i})$ and ${f_{i+1}}^{-1}(S_{j,i+1})$, where $S^3 \times D^1$ denotes the $j$-th copy. For the trivial smooth bundles given by the restrictions of $f_i$ and $f_{i+1}$ to the preimages, we can consider the images of suitable sections.
Let ${{\rm PD}}_i(a_j) \in H^4(M_{i})$ and ${{\rm PD}}_{i+1}(a_j) \in H^4(M_{i+1})$ denote the Poincar\'e duals to the images of the sections, which are $3$-dimensional spheres in the $7$-dimensional manifolds, respectively. For an integer $i>0$, assume that
the class ${\Sigma}_{j=1}^{a} {p}_{i+1,j} a_j \in H_3(W_{f_0};\mathbb{Z})$ is represented by $S_i$ where ${p}_{i+1,j}=0,1$. In this situation the square of the dual of ${\phi}_{2,i}(0 \oplus 1) \in H_2(M_{i+1})$ and ${\Sigma}_{j=1}^{a} {p}_{i+1,j} {{\rm PD}}_{i+1}(a_j)$ agree.
\end{enumerate}
Moreover, ${p}_{i_1,j}-{p}_{i_2,j} \leq 0$ for $i_1 \geq i_2>0$.
We can show the following theorem by applying a finite iteration of such operations starting from the original special generic map. The rigorous proofs are left to readers.

\begin{Thm}
\label{thm:7}
Let $A$ and $B$ be free finitely generated commutative groups whose ranks are $a$ and $b$, respectively. Let
$\{a_j\}_{j=1}^{a}$ and $\{b_j\}_{j=1}^{b}$ be bases of $A$ and $B$ respectively.
For each integer $1 \leq i \leq b$, let $\{p_{i,j}\}_{j=1}^{a}$ be a sequence of $0$ or $1$ such that $p_{i_1,j}-p_{i_2,j} \leq 0$ for $i_1 \geq i_2$.

In this situation, there exist an $7$-dimensional closed and simply-connected manifold $M$ and a fold map $f:M \rightarrow {\mathbb{R}^4}$ satisfying the following properties.
\begin{enumerate}
\item $H_2(M;\mathbb{Z})$ is isomorphic to $B$ and $H_3(M;\mathbb{Z})$ is isomorphic to $A$.
\item By universal coefficient theorem, $H^j(M;\mathbb{Z})$ is isomorphic to $H_j(M;\mathbb{Z})$. Via suitable isomorphisms ${\phi}_2:B \rightarrow H_2(M;\mathbb{Z})$ and ${\phi}_3:A \rightarrow H_3(M;\mathbb{Z})$, we can identify $H_2(M;\mathbb{Z})$ with $B$ and $H_3(M;\mathbb{Z})$ with $A$. Furthermore, we can define the duals ${a_j}^{\ast}$ and
${b_{j}}^{\ast}$ of ${\phi}_3(a_j)$ and ${\phi}_2 (b_j)$, respectively, and we can canonically obtain bases of
$H^3(M;\mathbb{Z})$ and $H^2(M;\mathbb{Z})$.
For the products, the following four hold.
\begin{enumerate}
\item The product of ${a_{j_1}}^{\ast}$ and ${a_{j_2}}^{\ast}$ vanishes for any $j_1$ and $j_2$.
\item The product of ${a_{i}}^{\ast}$ and ${b_{j}}^{\ast}$ is equal to $p_{j,i} {\rm PD}({\phi}_3(a_i))$ for any $i$ and $j$.
\item The product of ${b_{j_1}}^{\ast}$ and ${b_{j_2}}^{\ast}$ vanishes for any pair $(j_1,j_2)$ of distinct numbers.
\item The product of ${b_{i}}^{\ast}$ and ${b_{i}}^{\ast}$ is equal to ${\Sigma}_{j=1}^{a} p_{i,j} {\rm PD}({\phi}_3(a_j))$ where
${\rm PD}({\phi}_3(a_j))$ is the Poincar\'e dual to ${\phi}_3(a_j)$.
\end{enumerate}
\item $f {\mid}_{S(f)}$ is an embedding.
\item The index of each singular point is $0$ or $2$.
\item Preimages of regular values in the image of $f$ are diffeomorphic to $S^3$.
\end{enumerate}
\end{Thm}

As a key ingredient in rigorous proofs, we explain about the assumption $p_{i_1,j}-p_{i_2,j} \leq 0$. This respects the way we take embedded $3$-dimensional spheres $S_{j,i}$ and $S_{j,i+1}$ in the regular value sets of the maps and the homology classes represented by these embedded spheres in the images, which are represented as boundary connected sums of finitely many copies of $S^3 \times D^1$.
\begin{figure}
\includegraphics[width=50mm]{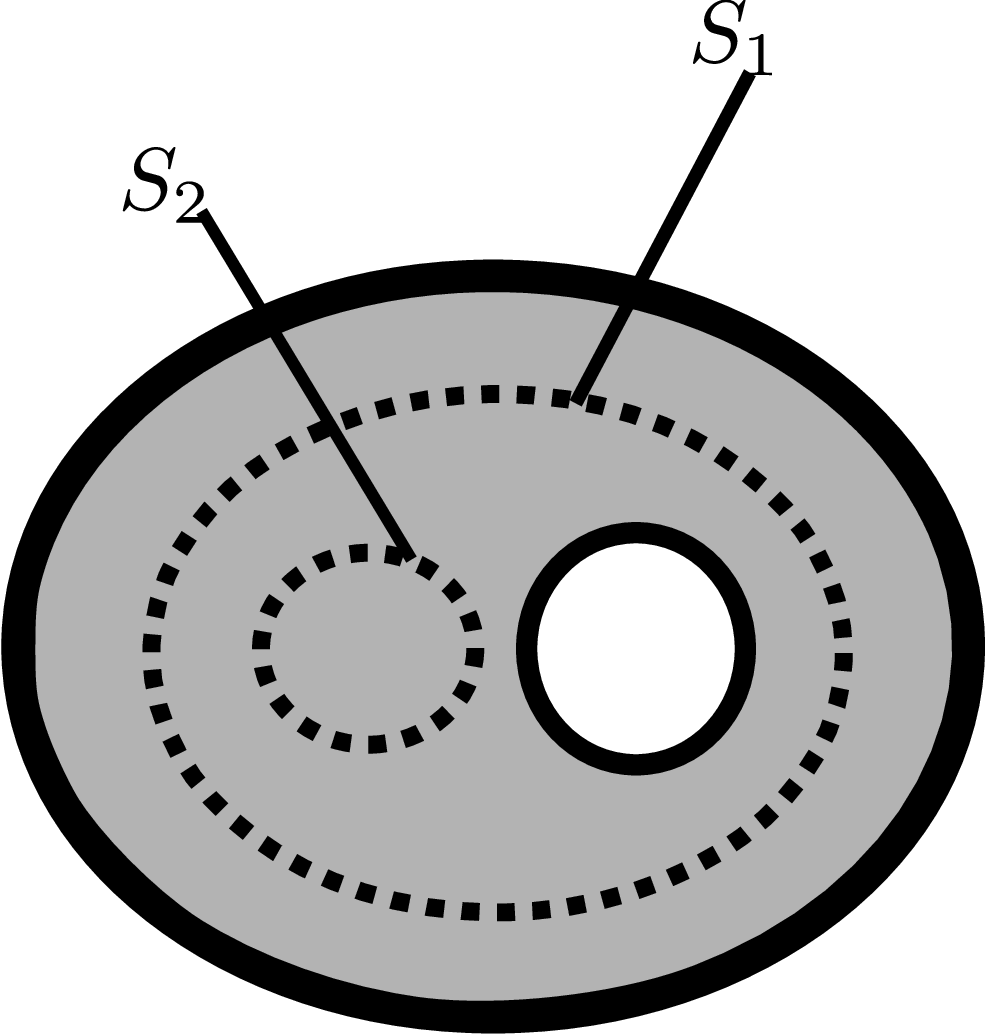}
\caption{The case $(a,b,p_{1,1},p{2,1})=(1,2,1,0)$ in Theorem \ref{thm:7} (An image).}
\label{fig:4}
\end{figure}

\subsection{A generalized version of Theorem \ref{thm:5} or \ref{thm:6}.}
Last, we present a generalized version of Theorem \ref{thm:5} (or \ref{thm:6}) except some properties such as ones on Pontryagin classes.

\begin{Thm}
\label{thm:8}
Let $k,m$ and $n$ be positive integers satisfying the relations $0<k \leq n-k<m-n$, $m-n \neq n$, $k+(m-n) \neq n$ and
$(n-k)+(m-n) \neq n$.
Let $A$, $B$ and $C$ be free finitely generated commutative groups whose ranks are $a$,$b$ and $c$, respectively. Let
$\{a_j\}_{j=1}^{a}$, $\{b_j\}_{j=1}^{b}$ and $\{c_j\}_{j=1}^{c}$ be bases of $A$, $B$ and $C$, respectively.
For each integer $1 \leq i \leq b$, let $\{a_{i,j}\}_{j=1}^{a}$ be a sequence of integers.

In this situation, there exist an $m$-dimensional closed and connected manifold $M$ whose Stiefel-Whitney classes and Pontryagin classes vanish and a fold map $f:M \rightarrow {\mathbb{R}^n}$ satisfying the following properties.
\begin{enumerate}
\item If $n-k \geq k>1$, then $M$ is simply-connected.
\item The homology group $H_{\ast}(M;\mathbb{Z})$ is free.
\item $H_k(M;\mathbb{Z})$ is isomorphic to $A$ if $k<n-k$ and $A \oplus B$ if $k=n-k$ and $H_n(M;\mathbb{Z})$ is isomorphic to $B \oplus C$.
\item By universal coefficient theorem, $H^j(M;\mathbb{Z})$ is isomorphic to $H_j(M;\mathbb{Z})$. We can take suitable monomorphisms ${\phi}_{A}:A \rightarrow H_{k}(M;\mathbb{Z})$ and ${\phi}_{B}:B \rightarrow H_{n-k}(M;\mathbb{Z})$ such that if $k=n-k$, then we can define an isomorphism by taking a suitable direct sum
${\phi}_{A} \oplus {\phi}_{B}:A \oplus B \rightarrow H_{k}(M;\mathbb{Z}) \oplus H_{n-k}(M;\mathbb{Z})$ and that if $k<n-k$, then these monomorphisms are isomorphisms. We can take an isomorphism
${\phi}_n:B \oplus C \rightarrow H_n(M;\mathbb{Z})$. Furthermore, we can define the duals ${a_j}^{\ast}$, ${b_{j,n-k}}^{\ast}$, ${b_{j,n}}^{\ast}$ and ${c_j}^{\ast}$ of ${\phi}_{A}(a_j)$, ${\phi}_{B}(b_j)$, ${\phi}_n((b_j,0))$ and ${\phi}_n((0,c_j))$, respectively, and we can canonically obtain bases of suitable submodules of $H^k(M;\mathbb{Z})$ and $H^{n-k}(M;\mathbb{Z})$ and a basis of $H^n(M;\mathbb{Z})$ by taking the suitable duals. Let the composition of ${\phi}_n$ with an isomorphism mapping the elements of the basis to their Poincar\'e duals be denoted by
${\phi}_{n,m-n}:B \oplus C \rightarrow H^{m-n}(M;\mathbb{Z})$. Let the composition of ${\phi}_{A}$ with an isomorphism mapping the elements of the basis to their Poincar\'e duals be denoted by
${\phi}_{A,m-k}:A \rightarrow H^{m-k}(M;\mathbb{Z})$. Let the composition of ${\phi}_{B}$ with an isomorphism mapping the elements of the basis to their Poincar\'e duals be denoted by ${\phi}_{B,m-n+k}:B \rightarrow H^{m-k}(M;\mathbb{Z})$.
For the products, we have the following cases for example.
\begin{enumerate}
\item 
\begin{enumerate}
\item The product of ${a_{j_1}}^{\ast}$ and ${a_{j_2}}^{\ast}$ and that of ${b_{j_1,n-k}}^{\ast}$ and ${b_{j_2,n-k}}^{\ast}$ vanish for any $j_1$ and $j_2$.
\item The product of ${a_{j_1}}^{\ast}$ and ${b_{j_2,n-k}}^{\ast}$ is $a_{j_2,j_1}{b_{j_2,n}}^{\ast} \in H^n(M;\mathbb{Z})$ for any $j_1$ and $j_2$.
\end{enumerate}
\item
\begin{enumerate}
\item The product of ${a_j}^{\ast}$ and ${\phi}_{n,m-n}((0,c))$ vanishes for any $j$ and $c \in C$.
\item The product of ${b_{j,n-k}}^{\ast}$ and ${\phi}_{n,m-n}((0,c))$ vanishes for any $j$ and $c \in C$.
\item The product of ${b_{j_1,n-k}}^{\ast}$ and ${\phi}_{n,m-n}((b_{j_2},0))$ vanishes for any pair $(j_1,j_2)$ of distinct numbers.
\item The product of ${a_{j_1}}^{\ast}$ and ${\phi}_{n,m-n}((b_{j_2},0))$ is
$a_{j_2,j_1}{\phi}_{B,m-n+k}(b_{j_2}) \in H^{m-n+k}(M;\mathbb{Z})$ for any $j_1$ and $j_2$.
\item The product of ${b_{i,n-k}}^{\ast}$ and ${\phi}_{n,m-n}((b_{i},0))$ is
${\Sigma}_{j=1}^{a} a_{i,j}{\phi}_{A,m-k}(a_j) \in H^{m-k}(M;\mathbb{Z})$.
\end{enumerate}
\item
\begin{enumerate}
\item The product of  ${a_j}^{\ast}$ and ${\phi}_{A,m-k}(a_j)$ forms a generator of $H^m(M;\mathbb{Z})$.
\item The product of  ${a_{j_1}}^{\ast}$ and ${\phi}_{A,m-k}(a_{j_2})$ vanishes for any pair $(j_1,j_2)$ of distinct numbers.
\item The product of  ${a_j}^{\ast}$ and ${\phi}_{B,m-n+k}(b)$ vanishes for any $j$ and $b \in B$.
\item The product of  ${b_{j,n-k}}^{\ast}$ and ${\phi}_{A,m-k}(a)$ vanishes for any $j$ and $a \in A$.
\item The product of  ${b_{j,n-k}}^{\ast}$ and ${\phi}_{B,m-n+k}(b_j)$ forms a generator of $H^m(M;\mathbb{Z})$.
\item The product of  ${b_{j_1,n-k}}^{\ast}$ and ${\phi}_{B,m-n+k}(b_{j_2})$ vanishes for any pair $(j_1,j_2)$ of distinct numbers.
\item The product of ${b_{j,n}}^{\ast}$ and ${\phi}_{n,m-n}((b_j,0))$ forms a generator of $H^m(M;\mathbb{Z})$.
\item The product of  ${b_{j_1,n}}^{\ast}$ and ${\phi}_{n,m-n}((b_{j_2},0))$ vanishes for any pair $(j_1,j_2)$ of distinct numbers.
\item The product of  ${b_{j,n}}^{\ast}$ and ${\phi}_{n,m-n}((0,c))$ vanishes for any $j$ and $c \in C$.
\item The product of  ${c_j}^{\ast}$ and ${\phi}_{n,m-n}((b,0))$ vanishes for any $j$ and $b \in B$.
\item The product of ${c_j}^{\ast}$ and ${\phi}_{n,m-n}((0,c_j))$ forms a generator of $H^m(M;\mathbb{Z})$.
\item The product of ${c_{j_1}}^{\ast}$ and ${\phi}_{n,m-n}((0,c_{j_2}))$ vanishes for any pair $(j_1,j_2)$ of distinct numbers.
\end{enumerate}
\end{enumerate}
\item $f {\mid}_{S(f)}$ is an embedding.
\item The index of each singular point is $0$ or $1$.
\item Preimages of regular values in the image of $f$ are diffeomorphic to $S^{m-n}$ or $S^{m-n} \sqcup S^{m-n}$.
\end{enumerate}
\end{Thm}
We can take $(m,n,k)=(7,4,2)$ and this yields Theorem \ref{thm:5} (or \ref{thm:6}) as a specific case of Theorem \ref{thm:8} except properties on
(non-vanishing) Stiefel-Whitney classes and Pontryagin classes. We can prove Theorem \ref{thm:8} almost similarly. The rigorous proofs are left to readers.

\begin{Rem}
\label{rem:2}
In \cite{kitazawa6}, homology groups and cohomology rings of {\it Reeb spaces} of fold maps are explicitly
studied. {\it Reeb spaces} of fold maps are defined as the spaces of all connected components of preimages of maps and defined as the quotient spaces of the manifolds.

Reeb spaces are in suitable cases inherit invariants of the manifolds such as homology groups, cohomology rings, and other kinds of invariants, considerably.

In the paper, the Reeb spaces of fold maps such that preimages of regular values are disjoint unions of standard spheres and that the restrictions to the singular sets are embeddings are systematically studied.
Proposition \ref{prop:6} and Theorems \ref{thm:5}--\ref{thm:7} are results on such maps. In \cite{kitazawa6}, the fact that for sufficiently high dimensional manifolds of the domains, the Reeb spaces inherit invariants before well is also presented for example.

However, only by discussions in the paper, it is difficult to understand higher degree parts of cohomology rings of the manifolds well for example. For homology groups, we can know well in cases where the dimensions of the manifolds are sufficiently high (or greater than or equal to the two times the dimensions of the manifolds of the targets). However, if the dimensions are not sufficiently high, then it is difficult to know the homology groups completely.

As a new work, in the present paper, homology groups and cohomology rings of manifolds admitting these fold maps are investigated including cases where the dimensions of the manifolds are not sufficiently high.
\end{Rem}


\section{Acknowledgement.}
\thanks{The author is a member of and supported by the project Grant-in-Aid for Scientific Research (S) (17H06128 Principal Investigator: Osamu Saeki)
"Innovative research of geometric topology and singularities of differentiable mappings''
(https://kaken.nii.ac.jp/en/grant/KAKENHI-PROJECT-17H06128/ : Principal Investigator is Osamu Saeki). The author would like to thank Osamu Saeki and all colleagues supported the project and the author, who give useful and interesting comments on the present paper.}

We declare that all data supporting our present study is in the present paper.

\end{document}